\documentclass[11pt, reqno]{amsart}
\usepackage{amssymb,latexsym,amsmath,amscd,amsthm,graphicx, color}
\usepackage{enumitem}
\usepackage[all]{xy}
\usepackage{pgf,tikz}
\usepackage{mathrsfs}
\usetikzlibrary{arrows}
\usepackage{pgfplots}
\usepackage{lineno}
\usepackage{soul}
\pgfplotsset{compat=1.15}
\usepackage[left=1 in, top=0.6 in, right=1 in, bottom=0.3 in]{geometry}
\usepackage{changepage}
\usepackage{float}        
\usepackage{subcaption}
\usepackage{geometry}
 \usepackage{algorithm}
 \usepackage{algpseudocode}
 \usepackage{bm}
 \usepackage{textcomp}
 \usepackage{booktabs}
 \usepackage[numbers]{natbib}
 
\pgfplotsset{compat=1.18} 

\allowdisplaybreaks
\makeatletter
\newcommand{\setword}[2]{%
	\phantomsection
	#1\def\@currentlabel{\unexpanded{#1}}\label{#2}%
}
\makeatother

\usepackage[linkcolor=blue, urlcolor=blue, citecolor=blue,
colorlinks, bookmarks]{hyperref}

\definecolor{uuuuuu}{rgb}{0.26666666666666666,0.26666666666666666,0.26666666666666666}
\definecolor{xdxdff}{rgb}{0.49019607843137253,0.49019607843137253,1.}
\definecolor{ffqqqq}{rgb}{1.,0.,0.}
\definecolor{ffqqqq}{rgb}{1.,0.,0.}
\definecolor{ffxfqq}{rgb}{1.,0.4980392156862745,0.}

\definecolor{uuuuuu}{rgb}{0.26666666666666666,0.26666666666666666,0.26666666666666666}
\definecolor{qqwuqq}{rgb}{0.,0.39215686274509803,0.}
\definecolor{zzttqq}{rgb}{0.6,0.2,0.}
\definecolor{xdxdff}{rgb}{0.49019607843137253,0.49019607843137253,1.}
\definecolor{qqqqff}{rgb}{0.,0.,1.}
\definecolor{cqcqcq}{rgb}{0.7529411764705882,0.7529411764705882,0.7529411764705882}
\definecolor{sqsqsq}{rgb}{0.12549019607843137,0.12549019607843137,0.12549019607843137}
\definecolor{uuuuuu}{rgb}{0.26666666666666666,0.26666666666666666,0.26666666666666666}
\definecolor{ffqqqq}{rgb}{1,0,0}
\definecolor{xdxdff}{rgb}{0.49019607843137253,0.49019607843137253,1}

\definecolor{yqqqyq}{rgb}{0.5019607843137255,0,0.5019607843137255}
\definecolor{qqqqff}{rgb}{0,0,1}
\definecolor{ffqqqq}{rgb}{1,0,0}
\definecolor{ffqqff}{rgb}{1,0,1}

\theoremstyle{plain}

\newtheorem{theorem}[subsection]{Theorem}

\newtheorem{corollary}[subsection]{Corollary}
\newtheorem{lemma}[subsection]{Lemma}
\newtheorem{definition}[subsection]{Definition}

\newtheorem{proposition}[subsection]{Proposition}

\theoremstyle{definition}

\newtheorem{remark}[subsection]{Remark}

\newcommand{\D}[1]{\mathbb{#1}}

\renewcommand{\for}[1]{\textbf{for}~#1}

\begin{document}

	\title{ Linear Stability analysis of the Lloyd algorithm on a circle}
 
 \author{$^1$Silpi Saha}
  \author{$^2$Sangita Jha}
	\author{$^3$Mrinal Kanti Roychowdhury}
 
		\address{$^{1, 2}$Department of Mathematics, National Institute of Technology Rourkela\\
		Rourkela, India 769008.}
 	\address{$^{3}$School of Mathematical and Statistical Sciences\\
		University of Texas Rio Grande Valley\\
		1201 West University Drive\\
		Edinburg, TX 78539-2999, USA.}

	\email{$^1$524ma1007@nitrkl.ac.in, $^2$jhasa@nitrkl.ac.in, $^3$mrinal.roychowdhury@utrgv.edu}

	
	\date{}

	\pagestyle{myheadings}\markboth{Silpi Saha, Sangita Jha, Mrinal Kanti Roychowdhury}{Linear Stability analysis of the Lloyd algorithm on a circle}
 \begin{abstract}
 Lloyd algorithm is the standard iterative method for computing quantizers and codebooks in source coding and vector quantization. In this article, we study the dynamical and stability properties of the Lloyd map on the unit circle $\mathbb S^1$ using von Mises distributions.  We construct the Lloyd iteration as a discrete dynamical system on the configuration space of ordered point sets modulo rotational symmetry. Also, we study the rotational equivarience of the Lloyd map. Further, we derive an explicit representation of the Jacobian matrix and prove that it possesses a circulant structure for the equally spaced configuration. Also, we study the bifurcation characteristics based on Lloyd map analysis.
In the end, we provide the numerical algorithms for stability diagrams, Lyapunov spectrum estimation, and residue analysis, purely for empirical visualization. Our results provide a dynamical systems framework for Lloyd quantization on $\mathbb S^1$ for studying stability properties.
 \end{abstract}
 \maketitle

\section{Introduction}

Stability analysis is a fundamental method in dynamical systems that determines whether a fixed point attracts or repels nearby trajectories as time evolves. The most common approach, known as linear stability analysis, approximates the local behaviour of a nonlinear system by examining how infinitesimally small perturbations grow or decay near a fixed point \cite{strogatz2018nonlinear}.
We achieve this by analyzing the Jacobian matrix, whose eigenvalues determine the stability \cite{strogatz2018nonlinear, grimshaw2017nonlinear}. Linear stability analysis is very powerful because it reduces complex nonlinear problems to simple matrix algebra, valid for small disturbances and provides local information. We study linear stability analysis coupled with mathematical theory of quantization in our current work.
Quantization plays a central role in information theory, signal processing, data compression, pattern recognition, and machine learning. The fundamental objective of quantization is to approximate a continuous source by a finite collection of representative points, called codepoints, so that the average distortion between the source and its quantized representation is minimized \cite{graf}. The mathematical theory of quantization has been extensively developed over the past several decades, with several scientific and engineering applications. Important contributions include the asymptotic theory of optimal quantizers, centroidal voronoi tessellations, and distortion minimization principles, and algorithm-based analysis \cite{gersho1991vector, gray1998quantization, voroo, LindeBuzoGray1980}. Among all quantization algorithms, the Lloyd algorithm occupies a particularly important position. The algorithm was originally introduced by Stuart P. Lloyd in his pioneering Bell Laboratories report and later published in reference~\cite{lloyd1982least}. The algorithm provides an iterative procedure for constructing locally optimal quantizers. At each iteration, the source space is partitioned into voronoi cells determined by the current codebook, and each codepoint is updated to the centroid of its corresponding cell. The method is a minimization technique. In this way, the distortion decreases monotonically and forms the basis of numerous algorithms in vector quantization, clustering, image compression, and unsupervised learning \cite{gersho1991vector, gray1998quantization, macqueen1967some}. The convergence theory of the Lloyd algorithm has been studied extensively. Classical results established conditions for convergence to stationary configurations and analyzed rates of convergence under regularity assumptions on the source distribution \cite{sabin1986global, kieffer1982exponential}. Subsequent investigations demonstrated that the convergence behaviour depends strongly on initialization, geometry of the source distribution, and the structure of the distortion distribution \cite{li2000convergence}. Nevertheless, the Lloyd iteration has not traditionally been viewed primarily as a stable fixed-point method that corresponds to stability analysis. In parallel, nonlinear dynamical phenomena such as bifurcations, oscillations, and chaotic behaviour have been observed in several iterative algorithms arising in statistics, optimization, and machine learning. Examples include expectation-maximization schemes, adaptive learning algorithms, and neural network training dynamics \cite{berlyand2016convergence, saad1996online}. These developments suggest that iterative optimization procedures may exhibit considerably stronger dynamical behaviour. We study quantization using the help of Lloyd algorithm and linear stability analysis, as it could open the door for new possiblities.

Our present work investigates the Lloyd algorithm on the circle $\mathbb{S}^1$ for nonuniform probability densities, with particular emphasis on the von Mises distribution, the canonical circular analogue of the Gaussian distribution. The concentration parameter $\kappa$ of the von Mises density naturally acts as a control parameter governing localization of the source distribution. A principal contribution of this work is the formulation of the Lloyd iteration as a discrete-time dynamical system on the manifold of ordered circular configurations modulo rotational symmetry. We derive explicit expressions for the Jacobian of the Lloyd map at symmetric fixed points, and show that the linearized dynamics possess a circulant and periodic tridiagonal structure generated by the rotational invariance of the problem. The stability criteria for equally spaced quantizers on the circle are also being analyzed here. Our analysis further establishes that rotational symmetry induces a neutral mode corresponding to rigid rotations of the entire configuration. The remaining spectrum determines the true dynamical stability of the quantizer after removing this neutral direction through quotient reduction. A detailed study of the eigenvalue spectrum yields the condition for the occurrence of flip (period-doubling) instabilities. We also present 3 different algorithms based on Lloyd iteration, corresponding to which we plot graphs for pure empirical visualization of our theoretical results. Overall, our work establishes a dynamical-systems framework for analyzing quantization algorithms on manifolds and reveals how symmetry, geometry, and source concentration jointly influence the stability properties of quantizers.

 The remainder of this paper is organized as follows. Section~\ref{sec:sec_2} establishes the necessary preliminaries, including the Lloyd algorithm on a circle and the von Mises distribution. Section~\ref{sec:sec_3} presents analysis of symmetric fixed points. Section~\ref{sec:sec_4} discusses about the linear stability analysis. Section~\ref{sec:sec_5} deals with the applications, algorithmic and graphical analysis. Section~\ref{sec:sec_6} discusses the implications of our findings and concludes with directions for future research.

\section{Preliminaries}
\label{sec:sec_2}
In this section, we review the fundamentals of quantization on a circle, structure of voronoi cells, and the Lloyd algorithm. We also discuss about circulant matrix, periodic tridiagonal matrices.
\begin{definition}
    A great circle on a sphere is formed where the sphere is cut by a plane passing through its center. It can also be described as a circle drawn on the sphere whose center is the same as the center of the sphere.
\end{definition}
\begin{definition}\cite{mk}
     Let $\D S^2:= \{ x \in \mathbb{R}^3 : \|x\| = 1 \}$
be the unit sphere, equipped with the \emph{geodesic distance} $d_G$. For two points $x, y \in \D S^2$, this distance is given by the central angle:
\begin{equation*}
d_G(x, y) = \arccos(\langle x, y \rangle),
\end{equation*}
where $\langle \cdot, \cdot \rangle$ is the standard inner product in $\mathbb{R}^3$. In spherical coordinates $(\phi, \theta)$, where $\phi \in [-\pi/2, \pi/2]$ is the latitude and $\theta \in [0, 2\pi)$ is the longitude, with
\begin{equation*}
x(\phi, \theta) = (\cos\phi \cos\theta, \cos\phi \sin\theta, \sin\phi).
\end{equation*}
The geodesic distance between $x_1=(\phi_1, \theta_1)$ and $x_2=(\phi_2, \theta_2)$ becomes
\begin{equation*}
d_G(x_1, x_2) = \arccos\left( \sin\phi_1 \sin\phi_2 + \cos\phi_1 \cos\phi_2 \cos(\theta_1 - \theta_2) \right).
\end{equation*}
\end{definition}
\begin{definition}
The distortion for a quantizer $Q$ is given by
\[
D(Q) = \sum_{j=1}^n \int_{R_j(Q)} d_G(\theta, q_j)^2 h(\theta) d\theta,
\]
where $d_G(\phi, \theta) = \min(|\phi-\theta|, 2\pi-|\phi-\theta|)$ is the geodesic distance. The Lloyd algorithm iteratively minimizes this distortion by alternating between voronoi partition updates and centroid updates.
\end{definition}

    In the above definition, we assume $\Gamma \subset \mathbb{S}^{2}$ be a great circle of length $L = 2\pi$, parameterized by the central angle $\theta \in [0,2\pi)$ measured from a reference point. For a fixed number of codepoints $n\geq 2$, a codebook is an ordered $n$-tuple $Q = (q_{1},q_{2},\ldots ,q_{n})\in \mathbb{T}^{n}$, where $\mathbb{T}^{n} = [0,2\pi)^{n}$ accounts for periodic boundary conditions. We assume without loss of generality that $0\leq q_{1}< q_{2}< \dots < q_{n}< 2\pi$ at initialization. We also assume that $h$ is our corresponding nonuniform probability density function and the associated Borel probability measure is $P$. $R_j(Q)$ is the voronoi cell, which corresponds to each codepoint. We define it more prominently in the upcoming definitions.

\begin{definition}\cite{decormo}
Let \(\alpha = \{a_1, \dots, a_n\} \subset \D S^2\) be a codebook. 
The sphere can be partitioned into \emph{spherical voronoi regions} defined by

\[
R(a_i \mid \alpha) = \Bigl\{ x \in \D S^2 : d_G(x, a_i) \leq d_G(x, a_j) \ 
\forall \, j \neq i, \; 1 \leq j \leq n \Bigr\}.
\]
\end{definition}
Each region \(R(a_i \mid \alpha)\) contains all points on the sphere whose geodesic distance to \(a_i\) is less than or equal to the distance to any other codepoint in \(\alpha\). We consider all voronoi cells are contained in arcs of length strictly less than $\pi$. This ensures that the intrinsic distance is realized by a unique minimizing geodesic between any two points in the same voronoi cell.

Hence, by using the above convention, we can write $d_G(\phi, \theta)=|\phi-\theta|$ (with appropriate identification modulo $2\pi$). For each codepoint $q_j$, define its voronoi cell
\[
R_j(Q) = \{\theta \in [0,2\pi): |\theta - q_j| \leq |\theta - q_k| \text{ for all } k \neq j\}.
\]
 On a circle, these cells are arcs whose boundaries are the midpoints between consecutive codepoints.
\begin{definition}
The one-dimensional sphere, denoted by \(\mathbb{S}^1\), is the set of all points in \(\mathbb{R}^2\) at unit distance from the origin:
\[
\mathbb{S}^1 = \{(x,y) \in \mathbb{R}^2 : x^2 + y^2 = 1\}.
\]
Equivalently, \(\mathbb{S}^1\) can be parametrized by an angular coordinate \(\theta \in [0, 2\pi)\) via the mapping:
\[
\theta \mapsto (\cos \theta, \sin \theta).
\]
\end{definition}
With this parametrization, \(\mathbb{S}^1\) is homeomorphic to the quotient space \([0, 2\pi] / \{0 \sim 2\pi\}\), i.e., a closed interval with its endpoints identified. The circle \(\mathbb{S}^1\) is a compact, connected, one-dimensional smooth manifold without boundary.

\begin{theorem}\cite{pressley2001elementary}
\label{thm:great_circle_isometry}
Let $\Gamma \subset \mathbb{S}^2$ be a great circle endowed with the intrinsic (geodesic) distance induced from $\mathbb{S}^2$. Then there exists a bijective map 
\[
\Phi : \Gamma \to \mathbb{S}^1
\]
which is an isometry, i.e.,
\[
d_{\mathbb{S}^2}(p,q) = d_{\mathbb{S}^1}(\Phi(p), \Phi(q)) \quad \text{for all } p,q \in \Gamma,
\]
where $d_{\mathbb{S}^2}$ denotes the great-circle distance on $\mathbb{S}^2$ and $d_{\mathbb{S}^1}$ denotes the standard arc-length distance on $\mathbb{S}^1$.
\end{theorem}

\begin{remark}
Any great circle $\Gamma \subset \mathbb{S}^2$ is the intersection of $\mathbb{S}^2$ with a plane through the origin. By applying a suitable rotation of $\mathbb{S}^2$, we may assume without loss of generality that $\Gamma$ coincides with the equator:
\[
\Gamma = \{(x,y,z) \in \mathbb{S}^2 : z = 0\}
= \{(x,y,0) : x^2 + y^2 = 1\}.
\]

Under this identification, $\Gamma$ is naturally identified with $\mathbb{S}^1$ via
\[
(x,y,0) \longmapsto (x,y) \in \mathbb{S}^1,
\]
which is a bijection preserving arc-length parametrization.

Moreover, the intrinsic (geodesic) distance on $\Gamma$ induced from $\mathbb{S}^2$ coincides with the standard arc-length distance on $\mathbb{S}^1$, given by
\[
d_G(\theta,\phi) = \min\bigl(|\theta - \phi|,\; 2\pi - |\theta - \phi|\bigr).
\]

Conversely, $\mathbb{S}^1$ may be viewed as a great circle embedded in $\mathbb{S}^2$. Therefore, throughout this work, we identify $\mathbb{S}^1$ with any fixed great circle $\Gamma \subset \mathbb{S}^2$, up to a rotation of $\mathbb{S}^2$ (i.e., an isometry).

In particular, both $\Gamma$ and $\mathbb{S}^1$ have total length $2\pi$ and identical intrinsic geometry.
\end{remark}
\begin{definition}\cite{li2000convergence, crutchfield1992evolution}
The Lloyd map $T:\mathbb T^n\to\mathbb T^n$
is defined by assigning to each voronoi cell its intrinsic centroid. Lloyd map operates by iteratively moving all quantizers to the centroid of their voronoi cells.
If every voronoi cell has arc length strictly smaller than $\pi$, then the intrinsic mean is uniquely defined and can be represented in local angular coordinates by 
\begin{equation*}
[T(Q)]_j
=
\frac{
\displaystyle
\int_{R_j(Q)}
\theta\, h(\theta)\,d\theta
}{
\displaystyle
\int_{R_j(Q)}
h(\theta)\,d\theta
}
\quad
(\mathrm{mod}\;2\pi).
\label{eq:lloyd_map}
\end{equation*}

Equivalently, the centroid may be expressed through the first circular moment. We define
\[
m_1(Q)
:=
\frac{
\displaystyle
\int_{R_j(Q)}
e^{i\theta} h(\theta)\,d\theta
}{
\displaystyle
\int_{R_j(Q)}
h(\theta)\,d\theta
}.
\]

Then $[T(Q)]_j
=
\arg\big(m_1(Q)\big)
\quad
(\mathrm{mod}\;2\pi)$.
\end{definition}

In the above definition, we consider $Q=(q_1,\dots,q_n)\in \mathbb T^n$ be an ordered configuration on the circle, and let $R_j(Q)$
denote the voronoi cell associated with $q_j$.
We also assume that the probability measure is absolutely continuous with respect to angular coordinate and has density $h(\theta)\in L^1(\mathbb S^1),
\quad
h(\theta)>0.$
 The map $T$ is smooth on regions of configuration space where the voronoi topology remains unchanged, and nonsmoothness may occur when voronoi boundaries collide or cell lengths approach $\pi$.

\subsection{The von Mises distribution}
To investigate the nonlinear dynamics of the Lloyd map, we require a family of probability measures that can be tuned continuously from nearly uniform to highly concentrated distributions on the circle. The von Mises distribution \cite{crutchfield1992evolution} provides the ideal parametric family for this purpose due to its simple functional form and well-characterized limiting behaviors:
$h(\theta; \kappa, \mu) = \frac{1}{2\pi I_0(\kappa)} e^{\kappa \cos(\theta - \mu)}, \quad \theta \in [0,2\pi), \label{eq:von_mises}$
where $\kappa \geq 0$ is the concentration parameter controlling the peakedness of the distribution, $\mu \in [0,2\pi)$ is the mean (or preferred) direction, and $I_0(\kappa)$ denotes the modified Bessel function of the first kind and order zero, which acts as the required normalizing constant to ensure $\int_0^{2\pi} h(\theta; \kappa, \mu)  d\theta = 1$.
\begin{remark}
 When $\kappa = 0$, $I_0(0) = 1$ and the density simplifies to $h(\theta; 0, \mu) = 1/(2\pi)$, recovering the uniform distribution on the circle \cite{fisher1993statistical}. As $\kappa \to \infty$, the distribution becomes increasingly concentrated around $\theta = \mu$, approaching a Dirac delta (point mass) at $\mu$ in the limit. For intermediate $\kappa > 0$, the distribution is unimodal, symmetric about $\mu$, and approximates a wrapped normal distribution, making it the circular analogue of the Gaussian.
\end{remark}
\begin{remark}
Due to the rotational invariance of the problem---i.e., $h(\theta + \phi; \kappa, \mu + \phi) = h(\theta; \kappa, \mu)$ for any $\phi \in [0, 2\pi)$---we may fix the mean direction $\mu = 0$ without loss of generality throughout our analysis. The von Mises distribution arises naturally in directional statistics as the maximum entropy distribution on the circle subject to a fixed first circular moment $\mathbb{E}[\cos(\theta - \mu)] = A(\kappa) := I_1(\kappa)/I_0(\kappa)$, where $I_1$ is the modified Bessel function of the first kind of order one \cite{jammalamadaka2001topics}. This maximum entropy property ensures it is the ``least informative'' distribution beyond the specified mean direction, analogous to the Gaussian on the line. It finds widespread use in applications involving directional data, such as wind direction in meteorology, animal orientation in biology, and neural tuning curves in neuroscience \cite{fisher1993statistical}.
\end{remark}

\begin{definition}\cite{gray2006toeplitz}
A matrix $C \in \mathbb{R}^{n \times n}$ is circulant if each row is a cyclic shift to the right of its previous row: 
\[
C = \begin{pmatrix}
c_0 & c_1 & c_2 & \cdots & c_{n-1} \\
c_{n-1} & c_0 & c_1 & \cdots & c_{n-2} \\
\vdots & \vdots & \vdots & \ddots & \vdots \\
c_1 & c_2 & c_3 & \cdots & c_0
\end{pmatrix}.
\]
\end{definition}
\begin{theorem}\cite{gray2006toeplitz, davis1979circulant}
For a circulant matrix $C$ with first row $(c_0, c_1, \ldots, c_{n-1})$, the eigenvalues are given by
\[
\lambda_m = \sum_{k=0}^{n-1} c_k e^\frac{-2\pi im k}{n}
.
\]
 The corresponding eigenvectors are the discrete Fourier modes
\[
v^{(m)}
=
\left(
1,
e^{-2\pi i m/n},
\dots,
e^{-2\pi i m(n-1)/n}
\right)^T.
\]
\end{theorem}

\begin{definition}\cite{ma}
A \text{ is a periodic tridiagonal matrix if } 
\[
 A = \begin{pmatrix}
a_1 & b_1 & 0 & \cdots & 0 & \beta \\
c_1 & a_2 & b_2 & \cdots & 0 & 0 \\
0 & c_2 & a_3 & \cdots & 0 & 0 \\
\vdots & \vdots & \ddots & \ddots & \ddots & \vdots \\
0 & 0 & \cdots & c_{n-2} & a_{n-1} & b_{n-1} \\
\alpha & 0 & \cdots & 0 & c_{n-1} & a_n
\end{pmatrix},
\]
where $\alpha, \beta, b_i, c_i, a_i \in \mathbb{F}$ (typically $\mathbb{R}$ or $\mathbb{C}$) and $\alpha, \beta \neq 0$ for non-trivial periodic coupling.
\end{definition}

Throughout this paper, we fix the number of codepoints $n$ and treat $\kappa$ as the primary parameter (without even considering von Mises density explicitly in some cases) to systematically explore the transition from disordered (low $\kappa$) to ordered (high $\kappa$) dynamical regimes.
\begin{definition}\cite{modeling, lei2024bifurcation}
    For a discrete-time dynamical system $\mathbf{x}_{k+1} = \mathbf{A} \mathbf{x}_k$, the fixed point is asymptotically stable if and only if all eigenvalues $\lambda$ of $\mathbf{A}$ lie inside the unit circle: $|\lambda| < 1.$
If any eigenvalue satisfies $|\lambda| > 1$, the system is unstable. If $|\lambda| = 1$ and all other eigenvalues satisfy $|\lambda| < 1$, marginal (or neutral) stability may occur.
\end{definition}
\begin{definition}
    A bifurcation is a qualitative change in the behaviour of a dynamical system caused by a smooth change in a parameter.
\end{definition}
\begin{remark}
    For a discrete-time dynamical system, the bifurcation associated with the appearance of the parameter value of -1, is called a flip (or period-doubling) bifurcation) \cite{lei2024bifurcation, kuznetsov2004}.
\end{remark}

\section{Fixed Points and the Symmetric Quantizer}
\label{sec:sec_3}
\begin{theorem}\label{thm:equivariance}
Let $T_h : \mathbb{T}^n \to \mathbb{T}^n$ be the Lloyd map associated with a probability density $h(\theta)$ on $\mathbb{S}^1$. For any rotation $R_\alpha : \mathbb{S}^1 \to \mathbb{S}^1$ is defined by
\[
R_\alpha(\theta) = \theta + \alpha \pmod{2\pi},
\]
define the rotated density as
\[
h_\alpha(\theta) := h(\theta - \alpha).
\]
Then the Lloyd map satisfies
\[
T_{h_\alpha}(R_\alpha Q) = R_\alpha T_h(Q) \quad \forall \hspace{0.1cm}Q \in \mathbb{T}^n, \hspace{0.1cm}\alpha \in \mathbb{S}^1.
\]
\end{theorem}

\begin{proof}
Let $Q = (q_1,\dots,q_n)$ and consider the rotated configuration $R_\alpha Q = (q_1+\alpha,\dots,q_n+\alpha)$. Since the geodesic distance on $\mathbb{S}^1$ is invariant under rotations \cite{pressley2001elementary}, the voronoi cells satisfy
\[
R_j(R_\alpha Q) = R_\alpha(R_j(Q)),
\]
up to sets of measure zero.

The centroid update is given by the circular mean (provided the integral is non-zero)
\[
[T_h(Q)]_j = \arg\left( \int_{R_j(Q)} e^{i\theta} h(\theta)\, d\theta \right).
\]

Applying $T_{h_\alpha}$ to $R_\alpha Q$, we obtain
\[
[T_{h_\alpha}(R_\alpha Q)]_j
= \arg\left( \int_{R_\alpha(R_j(Q))} e^{i\theta} h_\alpha(\theta)\, d\theta \right).
\]

Using the definition $h_\alpha(\theta) = h(\theta - \alpha)$ and the change of variables $\theta = \varphi + \alpha$, which preserves Lebesgue measure on $\mathbb{S}^1$, we get
\[
\int_{R_\alpha(R_j(Q))} e^{i\theta} h_\alpha(\theta)\, d\theta
= \int_{R_j(Q)} e^{i(\varphi+\alpha)} h(\varphi)\, d\varphi
= e^{i\alpha} \int_{R_j(Q)} e^{i\varphi} h(\varphi)\, d\varphi.
\]

Taking the argument yields
\[
[T_{h_\alpha}(R_\alpha Q)]_j
= \alpha + \arg\left( \int_{R_j(Q)} e^{i\varphi} h(\varphi)\, d\varphi \right)
= R_\alpha\bigl([T_h(Q)]_j\bigr),
\]
where we use that $\arg(e^{i\alpha} z) = \alpha + \arg(z) \pmod{2\pi}$ for $z \neq 0$.

Thus, we obtain
\[
T_{h_\alpha}(R_\alpha Q) = R_\alpha [T_h(Q)],
\]
which completes the proof.
\end{proof}

\medskip
\begin{remark}
In particular, if the density is rotationally invariant, i.e., $h(\theta) = h(\theta - \alpha)$ for all $\alpha$, then $h_\alpha = h$ and the Lloyd map satisfies
\[
T(R_\alpha Q) = R_\alpha [T(Q)].
\]
\end{remark}
\medskip
\begin{remark}
Although the von Mises density with mean $\mu = 0$ is not rotationally invariant, it is symmetric about the origin. This symmetry suggests the existence of symmetric codebook configurations $Q^* = \left(q_0^*, q_1^*, \ldots, q_{n-1}^*\right),
\quad q_j^* = \frac{2\pi j}{n}, \quad j = 0,1,\ldots,n-1$ \cite{saha2026optimal}.

\end{remark}

\begin{proposition}\label{prop:fixed_point}
Let $h(\theta;\kappa)$ be the von Mises density on $\mathbb{S}^1$ with mean direction $\mu = 0$. Assume that each voronoi cell $R_j(Q^*)$ is contained in an arc of length strictly less than $\pi$. Then the equally spaced configuration
\[
Q^* = \left(q_0^*, q_1^*, \ldots, q_{n-1}^*\right), \quad q_j^* = \frac{2\pi j}{n},
\]
is a fixed point of the Lloyd map $T$ on $\mathbb{S}^1$, i.e.,
\[
T(Q^*) = Q^*.
\]
\end{proposition}

\begin{proof}

Consider the codepoint $q_0^* = 0$. Its voronoi cell is
\[
R_0(Q^*) = \left(-\frac{\pi}{n}, \frac{\pi}{n}\right).
\]
The von Mises density satisfies
\[
h(-\theta;\kappa) = h(\theta;\kappa),
\]
so it is even. Hence, over the symmetric interval $R_0(Q^*)$, we compute
\[
\int_{R_0(Q^*)} e^{i\theta} h(\theta)\, d\theta
= \int_{-\pi/n}^{\pi/n} (\cos\theta + i\sin\theta) h(\theta)\, d\theta.
\]
The imaginary part vanishes since $\sin\theta\, h(\theta)$ is odd, while the real part is strictly positive. Therefore,
\[
[T(Q^*)]_0
= \arg\left( \int_{R_0(Q^*)} e^{i\theta} h(\theta)\, d\theta \right)
= 0 = q_0^*.
\]

\medskip

For the equally spaced configuration, each voronoi cell is a rotation of the reference cell:
\[
R_j(Q^*) = R_{2\pi j/n}(R_0(Q^*)).
\]

\medskip

Now, the voronoi cell $R_j(Q^*)$ is the interval 
\[
\left(\alpha_j - \frac{\pi}{n}, \alpha_j + \frac{\pi}{n}\right),
\]
which is symmetric about $\alpha_j$. Writing $\theta = \alpha_j + \psi$, we obtain
\[
\int_{R_j(Q^*)} e^{i\theta} h(\theta)\, d\theta
= e^{i\alpha_j} \int_{-\pi/n}^{\pi/n} e^{i\psi} h(\alpha_j+\psi)\, d\psi.
\]
Pairing the contributions at $\psi$ and $-\psi$, we can say that the integral inside is real and strictly positive. Hence the resulting vector is a positive multiple of $e^{i\alpha_j}$, and therefore
\[
\arg\left( \int_{R_j(Q^*)} e^{i\theta} h(\theta)\, d\theta \right) = \alpha_j.
\]
Thus,
\[
[T(Q^*)]_j = \alpha_j = q_j^*.
\]

\medskip

Since $[T(Q^*)]_j = q_j^*$ for all $j = 0,1,\ldots,n-1$, we conclude that
\[
T(Q^*) = Q^*.
\]
\end{proof}
\begin{proposition}
\label{prop:neutral_mode}

The Lloyd map $T_\kappa$ is rotationally equivariant:
\[
T_\kappa(Q+\phi\mathbf 1)
=
T_\kappa(Q)+\phi\mathbf 1,
\qquad
\phi\in\mathbb S^1,
\]
where $\mathbf 1=(1,\dots,1)^T$.

Consequently, the Jacobian at the symmetric fixed point satisfies
\[
DT_\kappa(Q^*)\,\mathbf 1
=
\mathbf 1.
\]

Therefore, the constant Fourier mode is an eigenvector of the Jacobian with eigenvalue
\[
\lambda_0=1.
\]
\end{proposition}

\begin{proof}

A rigid rotation of the configuration by $\phi$
rotates every voronoi cell by the same amount.
Since the density is defined intrinsically on $\mathbb S^1$,
the centroid of each rotated voronoi cell is precisely the rotated centroid of the original cell.

Hence
\[
T_\kappa(Q+\phi\mathbf 1)
=
T_\kappa(Q)+\phi\mathbf 1.
\]

Differentiating with respect to $\phi$ at $\phi=0$ yields (at the fixed point)
\[
DT_\kappa(Q^*)\,\mathbf 1
=
\mathbf 1.
\]

Thus $\mathbf 1$ is an eigenvector with eigenvalue $1$.
\end{proof}

\section{Linear Stability Analysis}
\label{sec:sec_4}
\begin{definition}

To assess the stability of the symmetric fixed point $Q^*$, we linearize the Lloyd map $T$ around $Q^*$. Let
\[
\delta Q^{(t)} = Q^{(t)} - Q^*
\]
denote a small perturbation. Then
\[
\delta Q^{(t+1)} = J(Q^*) \, \delta Q^{(t)},
\]
where the Jacobian matrix $J(Q^*) \in \mathbb{R}^{n \times n}$ is defined by
\[
J_{jk} = \frac{\partial [T(Q)]_j}{\partial q_k}\bigg|_{Q = Q^*}=DT(Q^*).
\]
\end{definition}
\medskip

\noindent
\begin{remark}
Under the assumption $h\in C^1(\mathbb S^1)$ and away from voronoi degeneracies,
the Lloyd map is continuously differentiable \cite{ll}. The derivatives can be computed using the Leibniz rule for integrals with moving boundaries.
\end{remark}
\medskip

Since each voronoi cell has length strictly less than $\pi$, the voronoi cells are given by
\[
R_j(Q) = (m_{j-1}, m_j), 
\quad 
m_j = \frac{q_j + q_{j+1}}{2}.
\]
We represent the centroid in local coordinates centered at $q_j$:
\[
[T(Q)]_j = q_j + C_j(Q),
\quad
C_j(Q) = \frac{N_j(Q)}{D_j(Q)},
\]
where
\[
N_j(Q) = \int_{m_{j-1}}^{m_j} (\theta - q_j)\, h(\theta)\, d\theta,
\quad
D_j(Q) = \int_{m_{j-1}}^{m_j} h(\theta)\, d\theta.
\]

\medskip

\begin{lemma}\label{lem:jacobian_structure_correct}
At the symmetric configuration $Q^*$, the Jacobian matrix $J(Q^*)$ is circulant and periodic tridiagonal. In particular,
\[
J_{jk} =
\begin{cases}
\alpha, & k = j, \\
\beta, & k = j\pm 1 \ (\mathrm{mod}\ n), \\
0, & \text{otherwise},
\end{cases}
\]
where
\begin{align*}
\alpha &= 1 - \frac{\pi}{nM}h\!\left(\frac{\pi}{n}\right), \hspace{0.2cm}\beta  = \frac{1}{2}\frac{\pi}{nM}h\!\left(\frac{\pi}{n}\right),
\end{align*}
and
\[
M = \int_{-\pi/n}^{\pi/n} h(\theta)\, d\theta.
\]
\end{lemma}

\begin{proof}

Since each voronoi cell has length strictly less than $\pi$, we may write
\[
[T(Q)]_j = q_j + C_j(Q),
\]
where
\[
C_j(Q) = \frac{N_j(Q)}{D_j(Q)},
\]
with
\[
N_j(Q) = \int_{m_{j-1}}^{m_j} (\theta - q_j)\, h(\theta)\, d\theta,
\quad
D_j(Q) = \int_{m_{j-1}}^{m_j} h(\theta)\, d\theta.
\]

Thus,
\[
J_{jk} = \delta_{jk} + \frac{\partial C_j}{\partial q_k}.
\]

\medskip

Now using the quotient rule, we obtain
\[
\frac{\partial C_j}{\partial q_k}
=
\frac{D_j \frac{\partial N_j}{\partial q_k}
-
N_j \frac{\partial D_j}{\partial q_k}}{D_j^2}.
\]

\medskip

Now, we apply the Leibniz rule:
\[
\frac{\partial}{\partial q_k} \int_{a(Q)}^{b(Q)} f(\theta)\, d\theta
=
f(b)\frac{\partial b}{\partial q_k}
-
f(a)\frac{\partial a}{\partial q_k}
+
\int_a^b \frac{\partial f}{\partial q_k}\, d\theta.
\]

For
\[
f(\theta) = (\theta - q_j)h(\theta),
\]
we have the derivative
\[
\frac{\partial f}{\partial q_k} = -\delta_{jk} h(\theta).
\]

Thus,
\begin{align*}
\frac{\partial N_j}{\partial q_k}
&=
(m_j - q_j) h(m_j)\frac{\partial m_j}{\partial q_k}
- (m_{j-1} - q_j) h(m_{j-1})\frac{\partial m_{j-1}}{\partial q_k}
- \delta_{jk} D_j, \\
\frac{\partial D_j}{\partial q_k}
&=
h(m_j)\frac{\partial m_j}{\partial q_k}
- h(m_{j-1})\frac{\partial m_{j-1}}{\partial q_k}.
\end{align*}

\medskip

Since,
\[
m_j = \frac{q_j + q_{j+1}}{2},
\]
we have
\[
\frac{\partial m_j}{\partial q_k} =
\begin{cases}
\frac{1}{2}, & k = j \text{ or } j+1, \\
0, & \text{otherwise},
\end{cases}
\]
and similarly, we can find the derivatives for $m_{j-1}$.

\medskip

Now at the symmetric configuration,
\[
m_j = q_j^* + \frac{\pi}{n}, 
\quad 
m_{j-1} = q_j^* - \frac{\pi}{n}.
\]

From Proposition~\ref{prop:fixed_point},
\[
[T(Q^*)]_j = q_j^* \quad \Rightarrow \quad C_j(Q^*) = 0.
\]

Hence,
\[
N_j(Q^*) = 0,
\]
and therefore
\[
\frac{\partial C_j}{\partial q_k}\Big|_{Q^*}
=
\frac{1}{D_j} \frac{\partial N_j}{\partial q_k}\Big|_{Q^*}.
\]

\medskip

By using,
\[
m_j - q_j^* = \frac{\pi}{n}, 
\quad
m_{j-1} - q_j^* = -\frac{\pi}{n},
\]
\medskip
We obtain for $k=j+1$
\[
\frac{\partial N_j}{\partial q_{j+1}}
=
\frac{\pi}{n} h(m_j)\cdot \frac{1}{2}
=
\frac{\pi}{2n} h(m_j),
\]
hence
\[
\frac{\partial C_j}{\partial q_{j+1}}
=
\frac{\pi}{2n} \frac{h(m_j)}{D_j}.
\]

\medskip
Finally for $k=j-1$
\[
\frac{\partial C_j}{\partial q_{j-1}}
=
\frac{\pi}{2n} \frac{h(m_{j-1})}{D_j}.
\]

\medskip

By rotational symmetry of $Q^*$, all rows are cyclic shifts of the row corresponding to $j=0$. Thus it suffices to evaluate at $j=0$.

Let
\[
M = \int_{-\pi/n}^{\pi/n} h(\theta)\, d\theta,
\]
and note that
\[
h(m_0) = h(\pi/n), \quad h(m_{-1}) = h(-\pi/n) = h(\pi/n).
\]

Thus,
\[
D_0 = M,
\]
and
\[
h(m_0) = h(m_{-1}) = h(\pi/n).
\]

\medskip

Substituting, we obtain: $\frac{\partial C_0}{\partial q_{\pm1}}
=
\frac{\pi}{2nM} h(\pi/n).$

Finally, $J_{0,\pm1} = \frac{\pi}{2nM}h(\pi/n)$. Now, from the construction of the Jacobian matrix, it is of the following form
\[
 J(Q^*) = \begin{pmatrix}
\alpha & \beta & 0 & \cdots & 0 & \beta \\
\beta & \alpha & \beta & \cdots & 0 & 0 \\
0 & \beta & \alpha & \cdots & 0 & 0 \\
\vdots & \vdots & \ddots & \ddots & \ddots & \vdots \\
0 & 0 & \cdots & \beta & \alpha & \beta \\
\beta & 0 & \cdots & 0 & \beta & \alpha
\end{pmatrix}.
\]
We can obviously say that $\beta = \frac{\pi}{2nM}h(\pi/n)$.
Now, by using Proposition~\ref{prop:neutral_mode}, we have $DT_\kappa(Q^*)\,\mathbf 1=\mathbf 1$.

Hence, the row sum $\alpha+2\beta=1$ (for each row). Therefore, $\alpha=1-2\beta=1-\frac{\pi}{nM}h(\pi/n)$.

\medskip

Finally, from the construction of the matrix we can say that every row is the cyclic shift to the right of its earlier row's elements. By this argument, we can say that the Jacobian matrix is circulant. Also we can see that the matrix is periodic tridiagonal.

\end{proof}
\subsection{Eigenvalues and Stability Criterion}

Since $J(Q^*)$ is circulant, its eigenvalues are given by the discrete Fourier modes:
\[
\lambda_m(\kappa) = \alpha(\kappa) + 2\beta(\kappa) \cos\left(\frac{2\pi m}{n}\right),
\qquad m = 0,\dots,n-1.
\]

For $m=0$, we obtain
\[
\lambda_0 = \alpha + 2\beta = 1,
\]
which corresponds to rotational invariance and represents a neutral mode.

\begin{theorem}\label{thm:stability_final}
The symmetric fixed point $Q^*$ is linearly stable if and only if
\[
\max_{1 \leq m \leq n-1} |\lambda_m| < 1.
\]

Assuming $\beta(\kappa) > 0$, the first instability occurs when
\[
\lambda_{m_*} = -1,
\]

where
\[
m_* =
\begin{cases}
n/2, & n \text{ even}, \\
\lfloor n/2 \rfloor, & n \text{ odd and large}.
\end{cases}
\]
\end{theorem}
\begin{proof}
Since $J(Q^*)$ is a real circulant matrix (by Lemma~\ref{lem:jacobian_structure_correct}), it is diagonalizable by the discrete Fourier transform, and its eigenvalues are given by $\lambda_m$.

The linearized system is stable if and only if the spectral radius satisfies
\[
\rho(J) = \max_m |\lambda_m| < 1,
\]
which is the standard discrete-time linear stability condition.

The mode $m=0$ corresponds to rotational invariance and satisfies $\lambda_0 = 1$, hence is neutral and does not affect stability.

Thus stability requires
\[
|\lambda_m| < 1 \quad \forall \hspace{0.1cm}m \neq 0.
\]

Now,
\[
\lambda_m = \alpha + 2\beta \cos\left(\frac{2\pi m}{n}\right).
\]

Since $\cos(\theta) \in [-1,1]$, the smallest eigenvalue occurs at the mode where cosine is minimized. Hence,
\[
\min_m \lambda_m = \alpha - 2\beta,
\]
which is attained at
\[
m_* =
\begin{cases}
n/2, & n \text{ even}, \\
\lfloor n/2 \rfloor, & n \text{ odd and large}.
\end{cases}
\]

Since $\lambda_0 = 1$ is fixed, stability requires that all other eigenvalues satisfy
\[
|\lambda_m| < 1, \quad m \neq 0.
\]

For $\beta > 0$, the eigenvalues are maximized at $m=0$, and hence
\[
\lambda_m \leq \lambda_0 = 1 \quad \text{for all } m.
\]
Thus no eigenvalue can exceed $1$, and loss of stability can only occur when the smallest eigenvalue decreases to $-1$.

Therefore, the first instability occurs when the most negative eigenvalue satisfies
\[
\lambda_{m_*} = -1.
\]
This completes the proof.
\end{proof}

\begin{corollary}
\label{cor:stability}
Assume $\beta > 0$. The symmetric fixed point $Q^*$ is linearly stable if and only if
\[
\frac{h(\pi/n)}{M}\cdot \frac{\pi}{n}
<
\frac{2}{1 - \cos\left(\frac{2\pi k^*}{n}\right)}.
\]
\end{corollary}

\begin{proof}
From the eigenvalue formula for the circulant Jacobian,
\[
\lambda_k = \alpha + 2\beta \cos\left(\frac{2\pi k}{n}\right),
\qquad k=0,\dots,n-1,
\]
and the identity
\[
\alpha + 2\beta = 1,
\]
we rewrite
\[
\lambda_k
= 1 - 2\beta\left(1 - \cos\left(\frac{2\pi k}{n}\right)\right).
\]

The mode $k=0$ gives $\lambda_0 = 1$, corresponding to rotational invariance, and is therefore neutrally stable.

Hence linear stability requires
\[
|\lambda_k| < 1 \quad \text{for all } k \neq 0.
\]

\medskip

Since $\beta > 0$, the function
\[
\cos\left(\frac{2\pi k}{n}\right)
\]
is decreasing on $k \in [0, n/2]$, and therefore $\lambda_k$ is minimized at
\[
k_* =
\begin{cases}
n/2, & n \text{ even}, \\
\lfloor n/2 \rfloor, & n \text{ odd and large}.
\end{cases}
\]

Thus the smallest eigenvalue is
\[
\lambda_{\min}
= 1 - 2\beta\left(1 - \cos\left(\frac{2\pi}{n}k_*\right)\right).
\]

Since $\lambda_k \le 1$ for all $k$, instability can only occur when the smallest eigenvalue move towards left of $-1$. Therefore, stability is equivalent to
\[
\lambda_{\min} > -1.
\]

Substituting, the value of $\lambda_{min}$,
\[
1 - 2\beta\left(1 - \cos\left(\frac{2\pi}{n}k_*\right)\right) > -1.
\]
Simplifying
\[
\beta\left(1 - \cos\left(\frac{2\pi}{n}k_*\right)\right) < 1.
\]

Now using $\beta = \frac{1}{2}\frac{h(\pi/n)}{M}\cdot \frac{\pi}{n},$
we obtain
\[
\frac{1}{2}\frac{h(\pi/n)}{M}\cdot \frac{\pi}{n}
\left(1 - \cos\left(\frac{2\pi}{n}k_*\right)\right) < 1.
\]

This yields,
\[
\frac{h(\pi/n)}{M}\cdot \frac{\pi}{n}
<
\frac{2}{1 - \cos\left(\frac{2\pi}{n}k_*\right)}.
\]
This completes the proof.
\end{proof}

\section{Applications}
\label{sec:sec_5}
\subsection{Flip-Bifurcation Analysis}
We regard the Lloyd map associated with the von Mises density as a smooth one-parameter family of maps
\[
T_\kappa : M_n \to M_n,
\qquad
\kappa \ge 0,
\]
where $M_n$ denotes the configuration space of ordered $n$-point configurations on $\mathbb S^1$ modulo global rotations.

Define
\[
\widetilde{M}_n
:=
\Big\{
Q=(q_1,\dots,q_n)\in [0,2\pi)^n
:
0\le q_1<q_2<\cdots<q_n<2\pi
\Big\}.
\]

We introduce the equivalence relation
\[
Q \sim Q+\phi \mathbf 1,
\qquad
\phi\in\mathbb S^1,
\]
where
\[
\mathbf 1=(1,\dots,1)\in\mathbb R^n.
\]

The reduced configuration space is then defined by
\[
M_n := \widetilde{M}_n/\sim,
\]
which removes the neutral direction associated with rigid rotations of the circle.

For every $\kappa\ge0$, the equally spaced configuration
\[
Q^*
=
\left(
\frac{2\pi j}{n}
\right)_{j=1}^n
\qquad (\mathrm{mod}\;2\pi)
\]
defines a symmetric fixed point of the Lloyd map:
\[
T_\kappa(Q^*) = Q^*.
\]

The map $T_\kappa$ depends smoothly on both the configuration $Q$ and the parameter $\kappa$. Indeed, the voronoi boundaries depend smoothly on the codepoints, and the centroid map is defined through integrals of smooth functions over smoothly varying intervals. In particular, the Jacobian matrix $J(\kappa)
=DT_\kappa(Q^*)$ depends smoothly in $\kappa$.
\medskip

\begin{theorem}
\label{thm:no_flip}
Let $T_\kappa : M_n \to M_n$
denote the Lloyd map on $\mathbb S^1$ associated with the von Mises density $h(\theta;\kappa)=e^{\kappa\cos\theta},$
and let $Q^*=(q_1^*,\dots,q_n^*)$ be the equally spaced configuration $q_j^*=\frac{2\pi j}{n}.$ Define $M(\kappa)
=
\int_{-\pi/n}^{\pi/n}
e^{\kappa\cos\theta}\,d\theta,$ and $F(\kappa)
=
\frac{\pi}{nM(\kappa)}
e^{\kappa\cos(\pi/n)}.$

Suppose
\[
F(\kappa)
<
\frac{2}{1-\cos(2\pi m_*/n)},
\]
where
\[
m_*=
\begin{cases}
n/2, & n \text{ even},\\
\lfloor n/2\rfloor, & n \text{ odd and large}.
\end{cases}
\]

Then every nontrivial eigenvalue of the Jacobian
\[
J(Q^*) = DT_\kappa(Q^*)
\]
satisfies $|\lambda_m(\kappa)|<1,
\qquad m=1,\dots,n-1.$
Consequently, no flip (period-doubling) bifurcation can occur.
\end{theorem}

\begin{proof}

From Lemma~\ref{lem:jacobian_structure_correct}, the Jacobian matrix at the symmetric configuration is circulant with first row
\[
(1-F(\kappa),\,F(\kappa)/2,\,0,\dots,0,\,F(\kappa)/2).
\]

Therefore, by the theory of circulant matrices, the eigenvectors are the discrete Fourier modes
\[
v^{(m)}
=
\left(
1,
e^{-2\pi i m/n},
\dots,
e^{-2\pi i m(n-1)/n}
\right)^T,
\]
and the corresponding eigenvalues are
\[
\lambda_m(\kappa)
=
1-F(\kappa)
+
F(\kappa)\cos\left(\frac{2\pi m}{n}\right),
\qquad
m=0,\dots,n-1.
\]

For $m=0$, we obtain
\[
\lambda_0
=
1-F(\kappa)+F(\kappa)
=
1.
\]

This eigenvalue corresponds to the rotational symmetry of the problem. Since, rigid rotations leave the voronoi structure invariant, so the Lloyd map is equivariant under the action of $\mathbb S^1$. Hence the direction
\[
(1,\dots,1)
\]
is neutral.

We therefore restrict attention to the quotient space modulo rotations, equivalently to the modes
\[
m=1,\dots,n-1.
\]

Since
\[
-1\le
\cos\left(\frac{2\pi m}{n}\right)
\le1,
\]
the smallest eigenvalue occurs for the mode minimizing the cosine term, namely
\[
m_*
=
\begin{cases}
n/2, & n \text{ even},\\
\lfloor n/2\rfloor, & n \text{ odd and large}.
\end{cases}
\]

Hence
\[
\lambda_m
\ge
\lambda_{m_*}
=
1-F(\kappa)
+
F(\kappa)\cos\left(\frac{2\pi m_*}{n}\right).
\]

The condition
\[
F(\kappa)
<
\frac{2}
{1-\cos(2\pi m_*/n)}
\]
implies
\[
F(\kappa)
\bigl(
1-\cos(2\pi m_*/n)
\bigr)
<2.
\]

Rearranging,
\[
1
-
F(\kappa)
+
F(\kappa)\cos\left(\frac{2\pi m_*}{n}\right)
>-1.
\]
Therefore, $\lambda_{m_*}>-1.$
Since all remaining eigenvalues satisfy $\lambda_m\ge\lambda_{m_*},$

it follows that
\[
\lambda_m>-1
\qquad
\text{for all }m=1,\dots,n-1.
\]

On the other hand,
\[
\lambda_m
=
1-F(\kappa)
+
F(\kappa)\cos\left(\frac{2\pi m}{n}\right)
<1
\]
for every nontrivial mode because
\[
\cos\left(\frac{2\pi m}{n}\right)<1
\qquad
(m\neq0).
\]

Hence $-1<\lambda_m<1,
\qquad
m=1,\dots,n-1$.
Therefore, all nontrivial eigenvalues lie strictly inside the unit disk. It ensures the stability of the system. Finally, a flip (period-doubling) bifurcation in a discrete dynamical system requires an eigenvalue to cross the point $-1$ on the unit circle. Since $\lambda_m(\kappa)>-1$
for all nontrivial modes, such a crossing cannot occur. Therefore no period-doubling instability arises at the symmetric configuration.
\end{proof}

\subsubsection{Stability Diagram Analysis}
We provide Algorithm~\ref{alg:stability_diagram} to empirically visualise the stability diagram. The algorithm is given as follows. Figure~\ref{fig:1} represents the stability diagram of the fixed points corresponding to the concentration parameter. Figure~\ref{fig:2} depicts the stability domain with respect to eigenvalue spectrum.
\begin{figure}[htbp]
    \centering
    \begin{subfigure}[b]{0.45\textwidth}
        \centering
        \includegraphics[width=\textwidth]{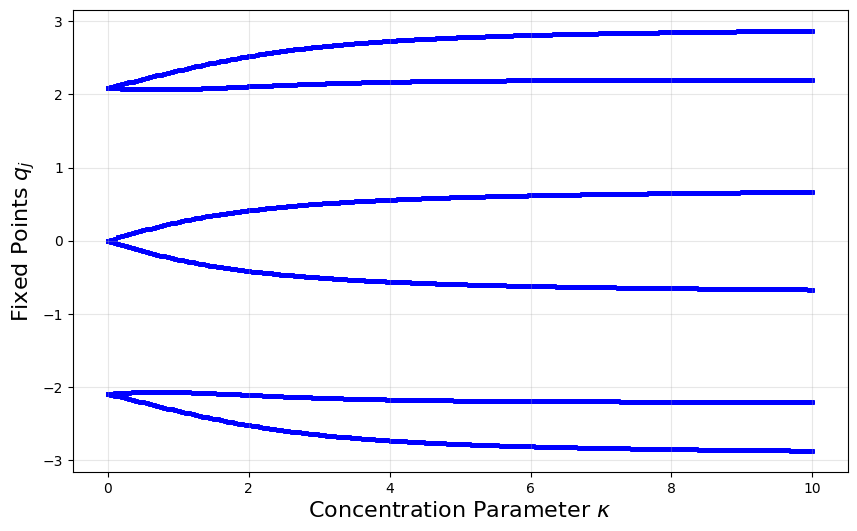} 
        \caption{Empirical observation for the stability of the fixed points with respect to the concentration parameter. Y axis represents the fixed points and X axis represents the concentration parameter.}
        \label{fig:1}
    \end{subfigure}
    \hfill
    \begin{subfigure}[b]{0.45\textwidth}
        \centering
        \includegraphics[width=\textwidth]{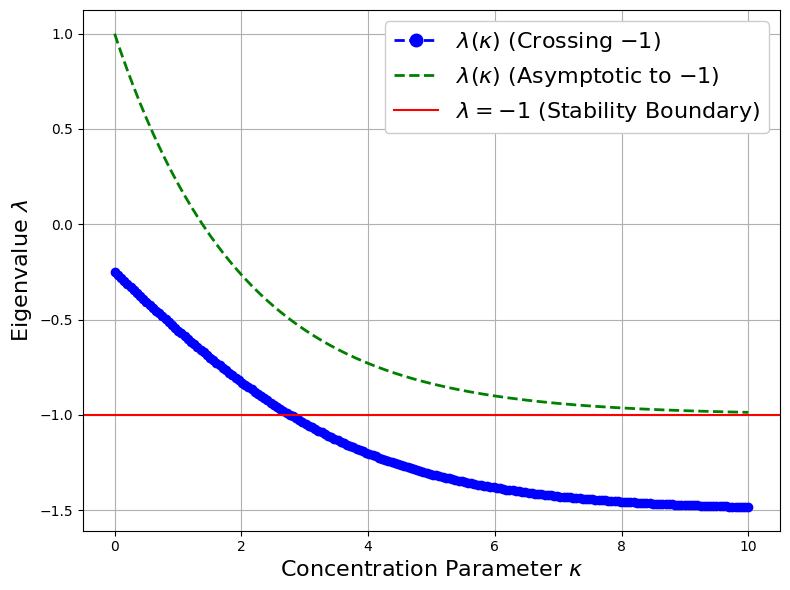} 
        \caption{Empirical observation of the stability criteria. The dashed curve shows the asymptotically stable nature of the eigenvalue. The solid line is the stability boundary. The thicker curve corresponds to instability.}
        \label{fig:2}
    \end{subfigure}
    \caption{Empirical Observation of Stability diagrams}
    \label{fig:both_figures}
\end{figure}

\begin{algorithm}[H]
\caption{Stability Diagram for the Lloyd Map on $\mathbb S^1$}
\label{alg:stability_diagram}
\begin{algorithmic}[1]

\Require Parameter interval $[\kappa_{\min},\kappa_{\max}]$, number of parameter samples $N_\kappa$, number of codepoints $n$, total iterations $N_{\mathrm{iter}}$, transient iterations $N_{\mathrm{trans}}$

\Ensure Arrays $(\mathcal K,\mathcal Q)$ for stability plotting

\vspace{0.2cm}

\State Compute parameter step size
\[
\Delta\kappa
=
\frac{\kappa_{\max}-\kappa_{\min}}{N_\kappa-1}
\]

\State Initialize empty arrays
\[
\mathcal K \gets \varnothing,
\qquad
\mathcal Q \gets \varnothing
\]

\vspace{0.2cm}

\for{$i=1,\dots,N_\kappa$}

    \State Set
    \[
    \kappa
    =
    \kappa_{\min}
    +
    (i-1)\Delta\kappa
    \]

    \State Initialize ordered random configuration
    \[
    Q^{(0)}
    \in [0,2\pi)^n
    \]

    \vspace{0.1cm}

    \for{$t=1,\dots,N_{\mathrm{iter}}$}

        \State Apply Lloyd update
        \[
        Q^{(t)}
        =
        T_\kappa(Q^{(t-1)})
        \]

        \State Wrap angles into $[0,2\pi)$:
        \[
        Q^{(t)}
        \gets
        Q^{(t)} \bmod 2\pi
        \]

        \State Sort configuration in increasing order

        \State Remove rotational drift:
        \[
        \bar q
        =
        \frac1n
        \sum_{j=1}^n q_j^{(t)}
        \]

        \[
        Q^{(t)}
        \gets
        Q^{(t)}-\bar q\,\mathbf1
        \]

        if {$t>N_{\mathrm{trans}}$}

            \for{$j=1,\dots,n$}

                \State Append $\kappa$ to $\mathcal K$

                \State Append $q_j^{(t)}$ to $\mathcal Q$

            end for

        end if

    end for

end for

\vspace{0.2cm}

\Return $(\mathcal K,\mathcal Q)$

\end{algorithmic}
\end{algorithm}

\begin{corollary} \label{cor:necessary_flip} Let $F(\kappa) = \frac{\pi}{nM(\kappa)}h(\frac{\pi}{n}, \kappa)$, where \[ M(\kappa) = \int_{-\pi/n}^{\pi/n} h(\theta;\kappa)\,d\theta. \] If the symmetric fixed point $Q^*$ loses stability through a flip (period-doubling) bifurcation, then the following condition hold\[ F(\kappa_c) = \frac{ 2 }{ 1-\cos\left(\frac{2\pi m_*}{n}\right) }, \] where \[ m_*= \begin{cases} n/2, & n \text{ even},\\ \lfloor n/2\rfloor, & n \text{ odd and large}. \end{cases} \]. \end{corollary}

\begin{proof}

From Theorem~\ref{thm:no_flip}, the eigenvalues of the Jacobian
\[
J(Q^*) = DT_\kappa(Q^*)
\]
are
\[
\lambda_m(\kappa)
=
1-F(\kappa)
+
F(\kappa)\cos\left(\frac{2\pi m}{n}\right),
\qquad
m=0,\dots,n-1.
\]

The mode $m=0$ satisfies
\[
\lambda_0=1,
\]
corresponding to the neutral rotational symmetry. Hence, the stability of the symmetric fixed point on the reduced configuration space $M_n$ is determined entirely by the nontrivial modes
\[
m=1,\dots,n-1.
\]

A flip (period-doubling) bifurcation in a discrete dynamical system can occur only if a real eigenvalue crosses the point $-1$ on the unit circle. Therefore, a necessary condition for a flip instability is the existence of some mode $m$ satisfying
\[
\lambda_m(\kappa_c)=-1
\]
for some critical parameter value $\kappa_c$.

Substituting the eigenvalue formula yields
\[
1-F(\kappa_c)
+
F(\kappa_c)\cos\left(\frac{2\pi m}{n}\right)
=
-1.
\]

Rearranging,
\[
F(\kappa_c)
\left(
1-
\cos\left(\frac{2\pi m}{n}\right)
\right)
=
2.
\]

Hence,
\[
F(\kappa_c)
=
\frac{
2
}{
1-\cos\left(\frac{2\pi m}{n}\right)
}.
\]

Among all Fourier modes, the smallest eigenvalue is obtained by minimizing the cosine term. Since
\[
\cos\left(\frac{2\pi m}{n}\right)
\]
attains its minimum at
\[
m_*=
\begin{cases}
n/2, & n \text{ even},\\
\lfloor n/2\rfloor, & n \text{ odd and large},
\end{cases}
\]
the first possible loss of stability must occur through this critical mode. This completes the proof.

\end{proof}

\begin{remark}
    The above corollary gives the condition for the occurrence of the instability. Furthermore, the value of $\kappa_c$ could be computed to check for any correlation that may exist by considering suitable values of $n$. Combining, those values, a proper conclusion could be drawn for correct scenario of the instability in the system.
\end{remark}

\subsection{Lyapunov Spectrum via QR Method}

To characterize the stability of the Lloyd map, we compute the Lyapunov exponents using the standard QR orthogonalization method for discrete dynamical systems. We have discussed an algorithm (Algorithm~\ref{alg:lyapunov_correct}) for empirical observation of the maximal Lyapunov exponent. Figure~\ref{fig:3} represents the empirical observation for the maximal Lyapunov exponent.

Let
\[
Q^{(t+1)} = T_\kappa(Q^{(t)})
\]
denote the Lloyd map on $(S^1)^n$. The linearized perturbation dynamics satisfies
\[
\delta Q^{(t+1)}
=
DT_\kappa(Q^{(t)})\,\delta Q^{(t)}.
\]

The Jacobian matrix is approximated numerically using centered finite differences together with the wrapped angular difference
\[
\operatorname{wrap}(\theta)
=
((\theta+\pi)\bmod 2\pi)-\pi.
\]

\begin{algorithm}[H]
\caption{Lyapunov Spectrum for the Lloyd Map}
\label{alg:lyapunov_correct}
\begin{algorithmic}[1]

\Require Concentration parameter $\kappa$, number of codepoints $n$, transient iterations $N_{\mathrm{trans}}$, sampling iterations $N_{\mathrm{iter}}$, finite-difference parameter $\varepsilon$

\Ensure Lyapunov exponents
\[
(\Lambda_1,\dots,\Lambda_n)
\]

\State Initialize random ordered configuration
\[
Q^{(0)} \in \widetilde M_n
\]

\State Initialize orthonormal basis
\[
U^{(0)} = I_n
\]

\State Initialize accumulators
\[
S_j = 0,
\qquad j=1,\dots,n
\]

\medskip

\for{$t=1$ to $N_{\mathrm{trans}}$}

    \[
    Q^{(t)} = T_\kappa(Q^{(t-1)})
    \]

end for

\medskip

\State Set
\[
Q \gets Q^{(N_{\mathrm{trans}})}
\]

\medskip

\for{$t=1$ to $N_{\mathrm{iter}}$}

    \State Approximate Jacobian by finite differences

    \for{$k=1$ to $n$}

        \[
        Q^{\pm} = Q \pm \varepsilon e_k
        \]

        \[
        J_t(:,k)
        =
        \frac{
        \operatorname{wrap}
        \left(
        T_\kappa(Q^{+})
        -
        T_\kappa(Q^{-})
        \right)
        }{
        2\varepsilon
        }
        \]

    end for

    \medskip

    \State Propagate tangent vectors
    \[
    V = J_t U^{(t-1)}
    \]

    \State Compute QR decomposition
    \[
    V = \mathcal{Q}\mathcal{R}
    \]

    \State Update orthonormal basis
    \[
    U^{(t)} = \mathcal{Q}
    \]

    \medskip

    \for{$j=1$ to $n$}

        \[
        S_j
        \gets
        S_j + \log |\mathcal{R}_{jj}|
        \]

    end for

    \medskip

    \State Advance trajectory
    \[
    Q \gets T_\kappa(Q)
    \]

end for

\medskip

\for{$j=1$ to $n$}

    \[
    \Lambda_j
    =
    \frac{S_j}{N_{\mathrm{iter}}}
    \]

end for

\Return
\[
(\Lambda_1,\dots,\Lambda_n)
\]

\end{algorithmic}
\end{algorithm}

Because the Lloyd map is rotationally invariant on $S^1$. A negative maximal Lyapunov exponent indicates stability in the system.
\begin{figure*}[htbp]
    \centering
    \includegraphics[width=0.8\linewidth]{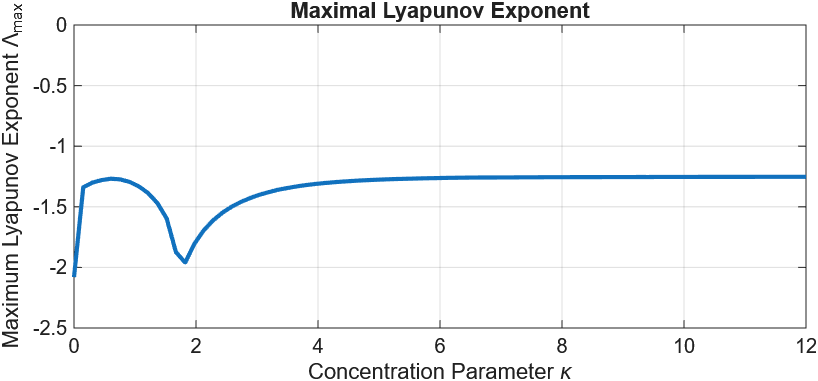}
   \caption{This is the empirical representation of the largest Lyapunov exponent. It always lies in the negative domain, and as the concentration parameter increases, it slowly converges. It suggests stability. }
    \label{fig:3}
\end{figure*}

\newpage
\subsection{Stability-based Lloyd Quantization}

The dynamical analysis of the von Mises density developed above suggests that the Lloyd map may exhibit persistent oscillatory or unstable behaviour for sufficiently large values of the concentration parameter $\kappa$. Motivated by this observation, we propose a stability-aware Lloyd iteration (Algorithm-\ref{alg:SALA}) that adaptively perturbs the codebook whenever non-convergent dynamics are detected. Figure~\ref{fig:4} represents the empirical observation of the residual decay corresponding to the increase in the number of iterations.
\begin{figure*}[htbp]
    \centering
    \includegraphics[width=0.8\linewidth]{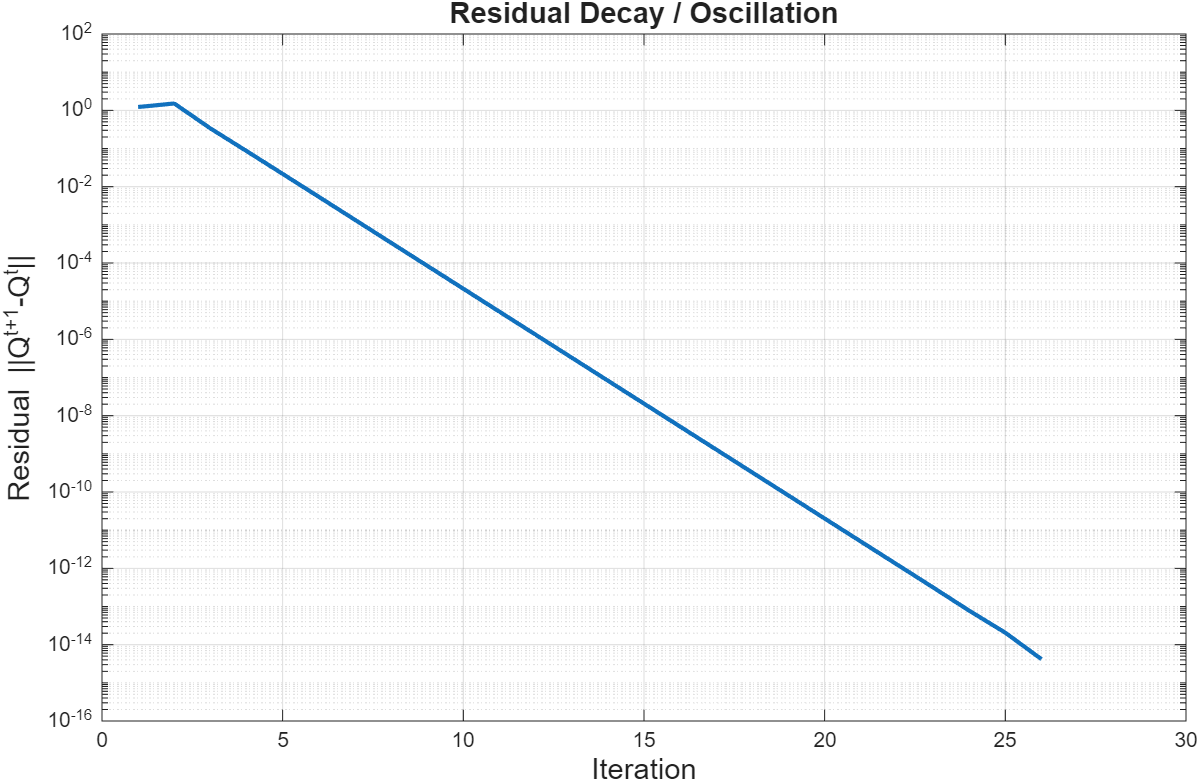}
   \caption{This is the empirical observation corresponding to Algorithm-\ref{alg:SALA}. It represents that as the number of iteration increases the residuals decreases, denoting stability.}
    \label{fig:4}
\end{figure*}

\begin{algorithm}[H]
\caption{Stability-Aware Lloyd Algorithm (SALA)}
\label{alg:SALA}
\begin{algorithmic}[1]

\Require Density $h(\theta)$, concentration $\kappa$,
tolerance $\varepsilon$,
oscillation threshold $\eta$,
maximum iterations $T_{\max}$

\Ensure Approximate stable codebook $Q$

\State Initialize ordered configuration
\[
Q^{(0)}=(q_1^{(0)},\dots,q_n^{(0)})\in\widetilde M_n
\]

\for{$t=0,\dots,T_{\max}-1$}

    \State Construct voronoi cells $R_j^{(t)}$

    \State Compute centroid updates
    \[
    q_j^{(t+1)}
    =
    \arg\min_{q\in\mathbb S^1}
    \int_{R_j^{(t)}}
    d_G(\theta,q)^2 h(\theta)\,d\theta
    \]

    \State Wrap and sort:
    \[
    Q^{(t+1)}
    \leftarrow
    \operatorname{sort}
    \big(
    Q^{(t+1)} \bmod 2\pi
    \big)
    \]

    \State Remove rotational drift:
    \[
    Q^{(t+1)}
    \leftarrow
    Q^{(t+1)}
    -
    \frac1n\sum_{j=1}^n q_j^{(t+1)}
    \]

    \State Compute residual
    \[
    r_t
    =
    \|Q^{(t+1)}-Q^{(t)}\|
    \]

    if {$r_t < \varepsilon$ for $L$ consecutive iterations}
        \State \Return $Q^{(t+1)}$
    
    end if

    if {$t\ge2$}

        \State Compute oscillation indicator
        \[
        \rho_t
        =
        \|Q^{(t+1)}-Q^{(t-1)}\|
        \]

        if {$\rho_t<\varepsilon$ and $r_t>\eta$}

            \State Approximate period-$2$ oscillation

            \State Apply perturbation
            \[
            Q^{(t+1)}
            \leftarrow
            Q^{(t+1)}+\delta\xi
            \]
            where $\xi$ is a small zero-mean random vector.

            \State Re-sort configuration.

        end if

    end if

end for

\State \Return $Q^{(T_{\max})}$

\end{algorithmic}
\end{algorithm}
\section{Conclusion}
\label{sec:sec_6}
In our work, we studied the Lloyd map on the unit circle $\mathbb S^1$ from the perspective of discrete dynamical systems and stability theory. We formulated Lloyd iterations corresponding to rotational symmetry, and analyzed the system by connecting quantization, circulant matrices, and nonlinear dynamics. We performed the explicit derivation of the Jacobian matrix, by considering the equally spaced configuration and showed that the Jacobian matrix is circulant. Our derivation allowed the eigenvalue spectrum to be computed exactly through discrete Fourier modes. 
We studied analytical stability conditions depending on the concentration parameter $\kappa$ of the von Mises density. We derived the condition for the onset of a flip (period-doubling) instability. We plotted stability diagrams, residual evolution diagrams, and Lyapunov spectra. The plots are entirely based on empirical observations from our theory to observe convergence and stability. Overall, our results demonstrated that the Lloyd algorithm on $\mathbb {S} ^1$ exhibits a nonlinear structure that can be studied by using linear stability analysis.
Our work provides a framework for studying the nonlinear dynamical behaviour of centroidal voronoi-type algorithms on compact manifolds. 

Furthermore, there can be several future possiblities based on our work. It includes extensions to higher-dimensional spheres $\mathbb S^d$, non-isotropic densities, manifold-valued quantization problems, stochastic Lloyd dynamics, and the study of nonlinear stability beyond the circular regime. Another research direction is the analysis of quantization dynamics on differential manifolds with nontrivial curvature or topology, where stronger bifurcation phenomena may arise. Furthermore, engineering-based applications are also possible. Circular signal quantization could be a vital area in the field of electrical and communication systems.

 \end{document}